\documentclass[12pt]{amsart}
\usepackage{amsmath,amssymb,amsbsy,amsfonts,amsthm,latexsym,mathabx,
            amsopn,amstext,amsxtra,euscript,amscd,stmaryrd,mathrsfs,
            cite,array,mathtools,enumerate}
\usepackage{soul}
\usepackage{url}
\usepackage[colorlinks,linkcolor=blue,anchorcolor=blue,citecolor=blue,backref=page]{hyperref}

\usepackage{color}

\usepackage[norefs,nocites]{refcheck}

\def\le{\leqslant}
\def\ge{\geqslant}
\def\leq{\leqslant}
\def\geq{\geqslant}
\usepackage{mathtools}
\usepackage{todonotes}

\newcommand\vol{\operatorname{vol}}

\usepackage{color}
\usepackage{float}

\hypersetup{breaklinks=true}

\usepackage[english]{babel}

\usepackage{enumitem}

\begin{document}

\newtheorem{theorem}{Theorem}
\newtheorem{lemma}[theorem]{Lemma}
\newtheorem{claim}[theorem]{Claim}
\newtheorem{cor}[theorem]{Corollary}
\newtheorem{prop}[theorem]{Proposition}
\newtheorem{definition}{Definition}
\newtheorem{question}[theorem]{Open Question}
\newtheorem{example}[theorem]{Example}
\newtheorem{rem}[theorem]{Remark}

\numberwithin{equation}{section}
\numberwithin{theorem}{section}
\numberwithin{table}{section}

\numberwithin{figure}{section}

\newcommand{\commL}[2][]{\todo[#1,color=orange!60]{L: #2}}
\newcommand{\commI}[2][]{\todo[#1,color=green!60]{I: #2}}
\newcommand{\commII}[2][]{\todo[#1,color=red!60]{I: #2}}

\newcommand{\ccr}[1]{{\color{red} #1}}
\newcommand{\ccm}[1]{{\color{magenta} #1}}
\newcommand{\ccc}[1]{{\color{cyan} #1}}
\newcommand{\cco}[1]{{\color{orange} #1}}
\newcommand{\ccg}[1]{{\color{gray} #1}}

\newcommand{\bflambda}{{\boldsymbol{\lambda}}}
\newcommand{\bfmu}{{\boldsymbol{\mu}}}
\newcommand{\bfxi}{{\boldsymbol{\xi}}}
\newcommand{\bfrho}{{\boldsymbol{\rho}}}

\renewcommand{\u}{\mathbf{u}}
\renewcommand{\v}{\mathbf{v}}
\newcommand{\SL}{\mathrm{SL}}
\newcommand{\GL}{\mathrm{GL}}
\def\ZK{\cO_\K}

 \newcommand{\F}{\mathbb{F}}
\newcommand{\K}{\mathbb{K}}
\newcommand{\R}{\mathbb{R}}
\newcommand{\D}[1]{D\(#1\)}
\def\scr{\scriptstyle}
\def\\{\cr}
\def\({\left(}
\def\){\right)}
\def\[{\left[}
\def\]{\right]}
\def\<{\langle}
\def\>{\rangle}
\def\fl#1{\left\lfloor#1\right\rfloor}
\def\rf#1{\left\lceil#1\right\rceil}
\def\le{\leqslant}
\def\ge{\geqslant}
\def\eps{\varepsilon}
\def\mand{\qquad\mbox{and}\qquad}

\def \e{\mathbf{e}}

\newcommand{\ord}{\mathrm{ord}}

\newcommand{\Fq}{\mathbb{F}_q}
\newcommand{\Fp}{\mathbb{F}_p}
\newcommand{\Disc}[1]{\mathrm{Disc}\(#1\)}
\newcommand{\Res}[1]{\mathrm{Res}\(#1\)}
\newcommand{\QQ}{\mathbb{Q}}
\newcommand{\ZZ}{\mathbb{Z}}

\newcommand\fp{\mathfrak{p}}
\newcommand\fq{\mathfrak{q}}

\newcommand{\p}{\mathfrak{p}}

\renewcommand{\Re}{\mathrm{Re}}
\def\GL{\mathrm{GL}}
\def\cA{{\mathcal A}}
\def\cB{{\mathcal B}}
\def\cC{{\mathcal C}}
\def\cD{{\mathcal D}}
\def\cE{{\mathcal E}}
\def\cF{{\mathcal F}}
\def\cG{{\mathcal G}}
\def\cH{{\mathcal H}}
\def\cI{{\mathcal I}}
\def\cJ{{\mathcal J}}
\def\cK{{\mathcal K}}
\def\cL{{\mathcal L}}
\def\cM{{\mathcal M}}
\def\cN{{\mathcal N}}
\def\cO{{\mathcal O}}
\def\cP{{\mathcal P}}
\def\cQ{{\mathcal Q}}
\def\cR{{\mathcal R}}
\def\cS{{\mathcal S}}
\def\cT{{\mathcal T}}
\def\cU{{\mathcal U}}
\def\cV{{\mathcal V}}
\def\cW{{\mathcal W}}
\def\cX{{\mathcal X}}
\def\cY{{\mathcal Y}}
\def\cZ{{\mathcal Z}}

\def\Q{\mathbb{Q}}
\def\Z{\mathbb{Z}}
\def\N{\mathbb{N}}

\newcommand{\sR}{\ensuremath{\mathscr{R}}}
\newcommand{\sDI}{\ensuremath{\mathscr{DI}}}
\newcommand{\DI}{\ensuremath{\mathrm{DI}}}

\newcommand{\Nm}[1]{\mathrm{Norm}_{\,\F_{q^k}/\Fq}(#1)}

\newcommand{\Tr}[1]{\mathrm{Tr}\(#1\)}
\newcommand{\rad}[1]{\mathrm{rad}(#1)}

\newcommand{\Orb}[1]{\mathrm{Orb}\(#1\)}
\newcommand{\aOrb}[1]{\overline{\mathrm{Orb}}\(#1\)}

\renewcommand{\v}{\mathrm{v}}

\def\vec#1{\mathbf{#1}}
\def\vh{\vec{h}}
\def\vv{\vec{v}}
\def\vu{\vec{u}}
\def\vw{\vec{w}}

\def\fu{\mathfrak{U}}

\title[Matrix  pseudorandom number generator]
{Distribution of recursive matrix pseudorandom 
number generator modulo prime powers}

 \author[L. M{\'e}rai]{L{\'a}szl{\'o} M{\'e}rai}
\address{L.M.: Johann Radon Institute for
Computational and Applied Mathematics,
Austrian Academy of Sciences, Altenberger Stra\ss e 69, A-4040 Linz, Austria and Department of Computer Algebra, E\"otv\"os Lor\'and University,  H-1117 Budapest, Pazm\'any P\'eter s\'et\'any 1/C, Hungary}
\email{laszlo.merai@oeaw.ac.at}

\author[I.~E.~Shparlinski]{Igor E. Shparlinski}
\address{I.E.S.: School of Mathematics and Statistics, University of New South Wales.
Sydney, NSW 2052, Australia}
\email{igor.shparlinski@unsw.edu.au}

\begin{abstract}
Given a matrix $A\in \GL_d(\Z)$. 
We study the pseudorandomness of vectors $\vu_n$ generated by a linear  recurrent 
relation of the form 
$$
\vu_{n+1} \equiv A \vu_n \pmod {p^t}, \qquad n = 0, 1, \ldots, 
$$
modulo $p^t$ with a fixed prime $p$ and  sufficiently large integer $t \ge 1$. 
 We study such sequences over very short segments of length which 
is not accessible via previously used methods. 
Our  technique  is based on the method of N.~M.~Korobov (1972) of 
 estimating double Weyl sums and a fully explicit 
 form of the Vinogradov  mean value theorem due to K.~Ford (2002). 
 This is combined with some ideas from the work of I.~E.~Shparlinski  (1978) which allows to
 construct polynomial representations of the coordinates of $\vu_n$
 and 
 control the $p$-adic orders of  their coefficients in polynomial representation. 
 \end{abstract} 


\pagenumbering{arabic}

\maketitle

\section{Introduction}

\subsection{Recursive congruential pseudorandom numbers}

Let $p$ be a fixed prime number (in particular, all implied constants are allowed to depend on $p$). 
Let $t\geq 1$ be an integer and let  $\cR_t = \Z/p^t\Z$ be the residue ring  modulo $p^t$, 
which we always assume to be represented by the set $\{0, \ldots, p^t-1\}$. 

Given a matrix $A\in \GL_d(\Z)$ we consider the sequence 
 \begin{equation}
\label{eq:Seq Vec}
\fu
=\(\vu_n\)_{n=0}^\infty,  \qquad \vu_n \in \cR_t^d, 
\end{equation}
of  $d$ dimensional vectors 
over $\cR_t$, generated iteratively by the  relation 
 \begin{equation}
\label{eq:u=Au}
\vu_{n+1} \equiv A \vu_n \qquad n = 0, 1, \ldots, 
\end{equation}
for some initial vector $\vu=\vu_0$, such that at least one component of $\vu_0$ is a unit in $\cR_t$. 
More explicitly, we can also write 
 \begin{equation}
\label{eq:u=An}
\vu_{n} \equiv A^n \vu_0 \qquad n = 0, 1, \ldots
\end{equation}

It is easy to see that the vectors~\eqref{eq:Seq Vec}  
satisfy a linear recurrent relation 
 \begin{equation}
\label{eq:LRS-Vec}
 \vu_{n+d} \equiv a_{d-1}  \vu_{n+d-1}  +  \ldots +  a_0 \vu_{n}  \pmod {p^t}, \qquad n = 0, 1, \ldots, 
\end{equation}
 where 
 \begin{equation}
\label{eq:CharPoly}
 f(X) = X^d - a_{d-1}  X^{d-1}  - \ldots  -  a_0 \in \Z[X]
\end{equation}
is  the {\it characteristic polynomial\/} of $A$.  In particular,  the components of vectors~\eqref{eq:Seq Vec}, 
more generally any sequence of the form $\vv A^n \vu_0$, satisfies a scalar analogue of~\eqref{eq:LRS-Vec}
(hereafter we always assume the vectors are oriented the way to make vector multiplication possible). 

In fact, our investigation of the vector sequences~\eqref{eq:Seq Vec}  is based on studying congruences 
and exponential sums with scalar linear recurrent sequences 
 \begin{equation}
\label{eq:Seq}
u=\(u_n\)_{n=1}^\infty \in \cR_t^\infty
\end{equation}
generated by the relation
 \begin{equation}
\label{eq:LRS}
 u_{n+d} \equiv a_{d-1}  u_{n+d-1}  +  \ldots +  a_0 u_{n}  \pmod {p^t},  \quad 0 \le u_n < p^t, 
\end{equation}
for $n = 0, 1, \ldots$, with some {\it initial values\/} $u_0, \ldots u_{d-1} \in \cR_t$.

Sequences generated as in~\eqref{eq:u=Au}, or in~\eqref{eq:Seq} or in~\eqref{eq:LRS} 
are common sources of pseudorandom vectors and have been studied in a number 
of works, see~\cite[Section~10.1]{Nied} and references therein. 
However, reasonably elementary methods used before, 
are not capable to provide any nontrivial information about the distribution of such sequences. 
The only exception is the result of~\cite{Shp}, which however applies only to the full period $\tau_t$ and 
implies that, under some natural conditions,  the distribution of the fractional parts $\{u_n/p^t\}$, 
$n =1, \ldots, \tau_t$, deviates from the uniform distribution by no more that $\tau_t^{-1/d+o(1)}$ 
as $t \to \infty$. See also~\eqref{eq:full per} for a formal statement.

Here we build on the ideas on~\cite{Kor} and~\cite{Shp}, but also enhance them with several  new arguments.

We note that our results are of similar strength as those obtained 
in~\cite{MeSh1} for the called  inversive generator. In particular, as in~\cite{MeSh1} 
we reduce our sum to a family of {\it Weyl sums\/} with polynomials of growing degree, 
thus, also as in~\cite{MeSh1}, the recent breakthrough in this direction, see, for example,~\cite{Bou3}, do not apply to our setting.
However, quite differently 
to~\cite{MeSh1}, for matrix generators, we do not control the $p$-divisibility 
of every coefficient of these polynomials. Instead we use some ideas of~\cite{Shp}
to investigate joint
$p$-divisibility in groups of $d$ consecutive coefficients.

\subsection{Our results}
\label{sec:results}

Following the standard approach to establishing
the uniformity 
of distribution results, we first  obtain an upper bound for the exponential sums.

Let $A\in \GL_d(\Z)$ and  $\vu,\vv\in\Z^d$ be nonzero vectors. We consider the sum
$$
S_{p^t}(N; \vu,\vv)=\sum_{n=0}^{N-1}\e \left(\frac{1}{p^{t}} \vv A^n \vu\right), 
$$
where, as usual, we denote $\e(z) = \exp(2\pi i z)$.

Using the method of Korobov~\cite{Kor} together with  the use of the  Vinogradov  mean value theorem in
the explicit form given by Ford~\cite{Ford}, 
we can estimate $S_{p^t}(N; \vu,\vv)$ for the values $N$ in the range~\eqref{eq:MS range}
\begin{equation}
\label{eq:MS range}
\tau_t \ge N \ge \exp\(c \(t \log p\)^{2/3}\)
\end{equation}
for some absolute constant $c>0$ depending only $p$ and  the characteristic 
polynomial~\eqref{eq:CharPoly}
of $A$, provided it satisfies some natural conditions.

We remark the range~\eqref{eq:MS range} covers the values of $N$ 
which are smaller than any power of the period $\tau_t$.

We recall our assumption that~\eqref{eq:LRS} is the shortest linear recurrence relation
satisfied by $(u_n)_{n=1}^\infty$, hence the polynomial $f$ in~\eqref{eq:CharPoly} is indeed
the characteristic (also sometimes called minimal) polynomial of this sequence. 

We say that a polynomial $f \in \Z[X]$ is {\it nondegenerate\/} if among its roots $\lambda_1, \ldots, \lambda_s$  and their 
nontrivial ratios $\lambda_i/\lambda_j$, $1 \le i< j \le s$, there are no roots of unity.

Throughout the paper we always use the parameter 
\begin{equation}
\label{eq:rho}
\rho = \frac{\log N}{t \log p}
\end{equation}
which controls the size of $N$ relative to the modulus $p^t$ on a logarithmic scale.

We also say that the pair of vectors $(\vu, \vv) \in \cR_t^d \times  \cR_t^d$ is {\it proper\/}, if the sequence
$u_n = \vv A^n \vu$ does not satisfy of any linear recurrence modulo $p$ of length less that $d$. 

Finally, we call  $\vu = (u_0,\dots, u_{d-1}) \in \Z^d$ \emph{$p$-primitive} if
$$
\gcd\(u_0, \ldots, u_{d-1}, p\) = 1.
$$

\begin{theorem}\label{thm:main} Assume  that the characteristic polynomial $f$ in~\eqref{eq:Seq} is nondegenerate,   
has no multiple roots modulo $p$ and $\gcd(a_0,p)=1$. 
Then, uniformly over pair of proper vectors $(\vu, \vv) \in \cR_t^d \times  \cR_t^d$, for $1 \le N \leq \tau_t$ we have
$$
\left|S_{p^t}(N; \vu,\vv) \right| \leq c  N^{1- \eta\rho^2 / (d^4 (\log d)^2)}, 
$$
where $\rho$ is given by~\eqref{eq:rho}, 
for some constant $c$ depending only on $f$ and $p$,  and 
some absolute constant $\eta>0$. 
\end{theorem}

We now  note that if $f$ is irreducible modulo $p$, they any    $p$-primitive  vectors $\vu,\vv \in \cR_t^d$  
form a proper pair. 

For a sequence $\fu=(\u_n)_{n= 0}^\infty$
defined by~\eqref{eq:u=Au} or more explicitly 
as in~\eqref{eq:u=An}, we denote by  $D_{\fu}(N)$ the  \emph{discrepancy}
of the vectors
\begin{equation}
\label{eq:LN-segment}
p^{-t}  \vu_n \in [0,1)^d, \qquad n = 1, \ldots, N,
\end{equation}
which 
 is defined by
\begin{equation}
\label{eq:Discr N}
D_{\fu}(N)=\sup_{J\subset[0,1)}\left|\frac{A_\fu(\cB,N)}{N}- \vol \cB \right|,
\end{equation}
where the supremum is taken over all boxes  $\cB$ of the form
$$
\cB = [\alpha_1, \beta_1] \times \ldots  \times [\alpha_d, \beta_d]\subseteq  [0,1)^d,
$$
and  $A_\fu(N,\cB)$ is the number of point~\eqref{eq:LN-segment}
which fall in $\cB$, 
and 
$$ 
\vol \cB = ( \beta_1-\alpha_1)   \cdots(\beta_d-\alpha_d)
$$ 
is the volume of $\cB$. 
The celebrated {\it Koksma--Sz\"usz inequality\/}~\cite{Kok,Sz} 
(see also~\cite[Theorem~1.21]{DrTi}) 
allows us to give an upper bound on the discrepancy $D_\fu(N)$ in terms of exponential sums
estimated in Theorem~\ref{thm:main}.

\begin{theorem}\label{thm:discrepancy}
Assume  that the characteristic polynomial $f$ in~\eqref{eq:Seq} is nondegenerate and irreducible modulo $p$. If $\vu$ is $p$-primitive,
then for $1 \le N \leq \tau_t$ we have 
$$
 D_{\fu}(N) \leq c_0  N^{-\eta_0 \rho ^2/(d^4 (\log d)^2)} (\log N)^d
$$
for some constant $c_0$ depending only on $f$ and $p$, and 
some absolute constant $\eta_0>0$. 
\end{theorem} 

Writing 
$$
N^{-\rho ^2} = \exp\(-\frac{(\log N)^3}{(t \log p)^2} \)
$$
we see that Theorems~\ref{thm:main} and~\ref{thm:discrepancy} are nontrivial in the range~\eqref{eq:MS range}.

\subsection{Related results}

We first recall that in the case of $d=1$
results of the similar strength as of 
Theorems~\ref{thm:main} and~\ref{thm:discrepancy}  are given by Korobov~\cite{Kor}. 

On the other hand, for arbitrary $d$
but only in the case of sums over the full period, that is, 
for $N = \tau_t$, the bound 
\begin{equation}
\label{eq:full per exp} 
\left|S_{p^t}(\tau_t; \vu,\vv) \right|  \le \tau_t^{1-1/d +o(1)}, \qquad t \to \infty,
\end{equation}
has been derived in~\cite{Shp}, which can be used to derive, 
under the conditions of  Theorem~\ref{thm:discrepancy}, 
the following bound on the discrepancy:
\begin{equation}
\label{eq:full per} 
D_{\fu}(\tau_t)  \le \tau_t^{-1/d +o(1)}, \qquad t \to \infty,
\end{equation}
see also~\cite[Equation~(7.8)]{EvdPSW}.

Furthermore for the counting function $A_\fu(N,\cB)$ in the definition of the discrepancy in~\eqref{eq:Discr N},
for  $N = \tau_t$ and cubes $\cB = [\alpha, \alpha+\delta]^d$ one has 
$$
 A_\fu(\tau_t,\cB) = O(\tau_t \delta)
 $$
 (with an implied constant depending only on $f$ and $p$), which is better than~\eqref{eq:full per} 
 for small $\delta =o( \tau_t^{-1/d})$, see~\cite[Theorem~7.6]{EvdPSW}. Furthermore, 
 for dimensional problems, or, equivalently for $\cB = [\alpha, \alpha+\delta]\times [0,1]^{d-1}$, 
 we have the following lower bound
 $$
 A_\fu(\tau_t,\cB) \ge  \frac{1}{d} \tau_t \delta + O(1),
 $$
see~\cite[Theorem~7.5]{EvdPSW}, which again is better than~\eqref{eq:full per} 
 for $\delta =o( \tau_t^{-1/d})$.

Furthermore, following the ideas of~\cite{MeSh2}, 
one can also study the distribution of vectors $\vu_\ell$ with prime indices $\ell$, 
see also~\cite{KeMeSh}.

Studying the distribution of sequences~\eqref{eq:u=Au} modulo a large prime is certainly a
very interesting question, which is also related to a discrete model of {\it quantum chaos\/}, 
see~\cite{Rud}. However in this setting, besides some very elementary bounds, stronger 
results are only know for $A \in \SL_d(\Z)$~\cite{OSV}, and especially for $d=2$, 
see~\cite{Bou1.5,KuRu,OSV}.

\subsection{Notation}
We recall that the notations $U\ll V$, and $V\gg U$ are 
equivalent to the statement that the inequality $|U |\leq cV$ holds with some
absolute constant $c>0$. 
 
We use the notation $\nu_p$ to denote the $p$-adic valuation, that is, for non-zero integers $a\in\Z$ we
 let $\nu_p(a)=k$ if $p^k$ is the highest power of $p$ which divides $a$, and $\nu_p(a/b)=\v_p(a)-\v_p(b)$ for $a,b\neq 0$.
 We also use $\nu_\fp$ for a similarly defined $\fp$-adic valuation when $\fp$ is a prime ideal of a number field. 
 
 For an integer ideal $\fq$  we write $\ord_\fq \gamma$ to denotes the order of an algebraic integer $\gamma$ (relatively prime 
 to $\fq$) modulo $\fq$. That is, $\ord_\fq \gamma$ is the smallest integer $t$ such that $\gamma^t \equiv 1 \pmod \fq$. 
 Similarly, we also use $\ord_\fq A$ for the order  of  a matrix $A$ modulo $\fq$ (provided 
 $\det A$ is relatively prime to $\fq$).

 We denote the cardinality of a finite set $\cS$ by $\#\cS$.
 
 We  use the notation $\lfloor x \rfloor$ and $\lceil x \rceil$ for the largest integer no larger than $x$ and the smallest integer no smaller than $x \in \R$, respectively. We then write $\{x\}=x- \lfloor x \rfloor \in [0,1)$ for the fractional part of $x$. 
 
 For a ring $\cR$ we use $\cR^{d\times d}$ to denote the ring of matrices over $\cR$. 
  
\section{Background on linear recurrence sequences} 
\subsection{Orders of algebraic integers modulo powers of prime id\-eals} 
Assume, that $\det A\not \equiv 0 \pmod p$ and $A$ is diagonalizable with eigenvalues $\gamma_1, \ldots, \gamma_d \in \overline{\Q}$. 
Let $\K=\Q(\gamma_1,\ldots, \gamma_d)$ be the splitting field of the characteristic polynomial $f$ of $A$ and let $\ZK$
be its ring of integers.  In particular, $\gamma_1,\ldots, \gamma_d \in \ZK$. 
Then for $\vu, \vv \in \Z^d$ we have
\begin{equation}
\label{eq:LRS=PowSum} 
u_n = \vv A^n \vu =  \frac{1}{\det A} \(\alpha_1\gamma_1^n+\ldots,+ \alpha_d\gamma_d^n\), \quad n\geq 0,
\end{equation}
for some $\alpha_1,\ldots, \alpha_d\in \ZK$, see~\cite[Section~1.1.7]{EvdPSW}.

We always assume, that $f$ has no multiple roots modulo $p$ (and thus its discriminant is not divisible by $p$).
Therefore $p$ is unramified and we can factor  $p$  in the product of several 
distinct prime ideals. 

We now fix a prime ideal  $\fp\subseteq \ZK$ which divides $p$ and 
use $\Z_\fp$ to denote the ring $\fp$-adic integers, see~\cite[Section~6.4]{Gouv} 
for a background. 

We observe that for any $w \in \Z$  and $s \in \N$, the divisibilities $p^s\mid w$ and $\fp^s\mid w$ 
are equivalent.  In particular
$$
\ord_{\fp^s} A = \ord_{p^s} A = \tau_s.
$$

\begin{lemma}\label{lemma:determinant}
Let $\fp\subseteq \ZK$ be a prime ideal and  let $M\in\mathrm{GL}_d( \Z_\fp)$. Assume that  $\vu\in \Z_\fp^{d}$
is such that $\vu \not \equiv (0,\ldots, 0) \pmod \fp$. 
If $M\vu\equiv (0,\ldots, 0) \pmod {\fp^w}$, then $\det M \equiv 0 \pmod {\fp^w}$.
\end{lemma}

\begin{proof} 
Write $M^{-1} =\( \det M\)^{-1} M^{*} $ with $M^{*}\in  \Z_\fp^{d \times d}$. If  
$$M\mathbf{u}\equiv (0,\ldots, 0) \pmod {\fp^w},
$$ 
then 
$$  \det M  \cdot \vu = M^* M \vu  \equiv (0,\ldots, 0) \pmod {\fp^w}.
$$ As $\vu$ has a component which is non-zero modulo $\fp$, we conclude that $\det M \equiv 0 \pmod {\fp^w}$.
\end{proof}

For $\gamma,\lambda \in \ZK$, which are relatively prime to a prime ideal $\fp$ we define $\tau_s (\gamma, \lambda)$ to be the minimal positive integer $t$ such that
$$
\gamma^t \equiv \lambda^t \pmod {\fp^s}.
$$ 
It is easy to see that $\tau_1 (\gamma, \lambda)$ is relatively prime to $p$. 

If $\lambda =1$, we write 
$$
\tau_s(\gamma)=\tau_s(\gamma,\lambda).
$$ 

\begin{lemma} 
\label{lem:per growth}
Let $p$ be a rational unramified prime and $p\ZK \subseteq \fp$. 
Let $\gamma,\lambda \in \ZK$ such that  $\gamma  ^ t \neq \lambda^t$ for any $t\neq 0$.
Write 
$$
\tau_{*} (\gamma,\lambda) = 
\begin{cases}
\tau_1(\gamma,\lambda) & \text{if $p\neq 2$},\\
\tau_2(\gamma,\lambda) & \text{if $p= 2$},
\end{cases}
$$ 
and let 
$$
\beta_{\fp} (\gamma, \lambda) = \nu_{\fp}\(\gamma^{\tau_{*} (\gamma,\lambda) } - \lambda^{\tau_{*} (\gamma,\lambda) } \).
$$ 
Then for $s\geq 2$, we have
$$
\tau_s (\gamma, \lambda)
=
\begin{cases}
\tau_{*} (\gamma,\lambda)  & \text{if } s\leq  \beta_{\fp} (\gamma, \lambda), \\
\tau_{*} (\gamma,\lambda) p^{ s-\beta_{\fp} (\gamma, \lambda)}  & \text{if } s\geq \beta_{\fp} (\gamma, \lambda).
\end{cases}
$$
\end{lemma}

\begin{proof}
It is enough to show that for $s\geq \beta_\fp(\gamma,\lambda)$ we have
$$
\nu_\fp((\gamma/\lambda)^{\tau_s(\gamma,\lambda)}-1)=s, 
$$
and thus we can assume $\lambda =1$ without loss of generality.

If
$$
\gamma^{\tau_s(\gamma)} -1 = \vartheta_s, \quad \nu_{\fp}(\vartheta_s) =s,
$$
then
$$
\gamma^{p\cdot \tau_s(\gamma)}=\left(1+ \vartheta_s \right)^p = 1 + \sum_{i=1}^{p-1}\binom{p}{i} \vartheta_s^{i} + \vartheta_s^{p}.
$$
Here
$$
\nu_\fp \left( \vartheta_s^{p}  \right)=p\cdot s \mand
\nu_\fp \left( \binom{p}{i} \vartheta_s^{i}  \right)=i\cdot s+1 \text{ for }  1\leq i<p.
$$
Since $s>1$ if $p=2$,  we see that 
$$
\nu_\fp \left( \binom{p}{i} \vartheta_s^{i}  \right)>s+1 \quad \text{for $2\leq i\leq p$}
$$
and thus $\nu_{\fp}\left( \gamma^{p\cdot \tau_s(\gamma)}-1 \right)=s+1$. 
\end{proof}

It follows from Lemma~\ref{lem:per growth}, that there are $s_*$,  $\beta_*$ and $\tau_*$ such that 
\begin{equation}\label{eq: tau*}
\tau_s = \ord _{\fp^s}A=  \tau_{*} p^{s- \beta_*}  \quad \text{with } \gcd \(\tau_{*} , p\) =1 \quad \text{for }  s \ge s_*.
\end{equation}  

Let $s\geq  s_*$.  Put
$$
\gamma_{i}^{\tau_s} = 1+\vartheta_i, \qquad i =1, \ldots, d.
$$
We observe that 
$$
\vartheta_i\in \fp^{s}, \quad i =1, \ldots, d, \mand  \vartheta_j\not \in \fp^{s+1 } \quad \text{for some}\ j \in \{1, \ldots, d\}. 
$$

\subsection{Polynomial representation of linear recurrence sequneces}
We see from~\eqref{eq:LRS=PowSum}  that  for any integers $m, n \ge 0$ and $s \ge 1$ we have 
\begin{equation}\label{eq:repr}\begin{split}
 u_{n+\tau_s m}& = \frac{1}{\det A}  \sum_{i=1}^d \alpha_i \gamma_{i}^{n+\tau_s m}=\frac{1}{\det A} \sum_{i=1}^d \alpha_i \gamma_{i}^{n}(1+\vartheta_i)^{m}\\  
 &=\frac{1}{\det A}  \sum_{i=1}^d \alpha_i \gamma_{i}^{n} \sum_{j=0}^{m}\binom{m}{j} \vartheta_i^{js}
    =\frac{1}{\det A} \sum_{j=0}^m   h_{n,j}  p^{js} \binom{m}{j}
 \end{split} 
\end{equation}
with
\begin{equation}\label{eq:h}
h_{n,j}= p^{-js} \sum_{i=1}^d \alpha_i \gamma_i^n \vartheta_i^j\in\Z.
\end{equation} 
Indeed, it is trivial that $\fp^{js}$ divides the inner sum in~\eqref{eq:h}
 and since $p$ is not ramified, so does $p$. Since clearly $h_{n,j} \in \Q$ we
 can now conclude that in fact  $h_{n,j} \in \Z$.

\begin{lemma}
\label{lem:nondiv} 
Let $h_{n,j}$ ($1\leq j \leq m$) be defined by~\eqref{eq:h} with  $s\geq  s_*$.  
Write $\gamma_0=1$ and let 
\begin{equation}\label{eq:w}
w= \frac{d(d+1)}{2}\max_{1 \le i < j\le d}\left\{\beta_\fp(\gamma_i,\gamma_j)-\beta_*, 0 \right\} +1.
\end{equation}
Then for any integer $j\in [0, m-d+1]$, we have  
$$
\(h_{n,j},\ldots,h_{n,j+d-1}\)\not \equiv (0,\ldots, 0) \pmod {p^{w}}.
$$ 
\end{lemma}

\begin{proof}
Suppose, that 
$$
\(h_{n,j},\ldots,h_{n,j+d-1}\)\equiv (0,\ldots, 0) \pmod {\fp^{w}},
$$
which, as we have mentioned, is equivalent to a similar congruence modulo $p^w$. 

Let $\pi$ be a uniformizer of the maximal ideal $\fp$ in $\Z_\fp$. As $\vartheta_i\in \fp^s$ we can write
$$
\widebar{\vartheta}_i = \pi^{-s}\vartheta_i \in \Z_\fp , \quad 1\leq i \leq d.
$$

Let
$$
\varTheta= 
\begin{pmatrix}
\widebar{\vartheta}_1^{j}& \ldots, & \widebar{\vartheta}_d^{j}\\
\vdots & \ddots & \vdots \\
\widebar{\vartheta}_1^{j+d-1} & \ldots, & \widebar{\vartheta}_d^{j+d-1}
\end{pmatrix}. 
$$
Then  
\begin{align*}
 \varTheta  \(\alpha_1\gamma_1^n,\ldots, \alpha_d\gamma_d^n\)  & = \( (p/\pi)^{js} h_{n,j},\ldots,  (p/\pi)^{(j+d-1)s} h_{n,j+d-1}\)\\
 &  \equiv \(0,\ldots, 0\)\pmod {\fp^w}.
\end{align*}

As $\(\alpha_1\gamma_1^k,\ldots, \alpha_d\gamma_d^k\)\not \equiv 0 \pmod \fp$, we have by Lemma~\ref{lemma:determinant} that 
\begin{align*}
\det \varTheta & = \widebar{\vartheta}_1^{j} \ldots \widebar{\vartheta}_d^{j} \prod_{0\leq i_1<i_2 \leq d} \(\widebar{\vartheta}_{i_1} -\widebar{\vartheta}_{i_2}\)\\
& =\widebar{\vartheta}_1^{j} \ldots \widebar{\vartheta}_d^{j}\prod_{0\leq i_1<i_2\leq d} \frac{\gamma_{i_1}^{\tau_s}-\gamma_{i_2}^{\tau_s}}{\pi^s}\equiv 0 \pmod {\fp^{w}},
\end{align*}
where we put $\gamma_0=1$, $\vartheta_0=0$. 

Thus 
$$
\gamma_{i_1}^{\tau_s}-\gamma_{i_2}^{\tau_s} \equiv 0 \pmod { \fp^{s + \rf{\frac{2w}{d(d+1)} } }}
$$
for some $1 \le i_1<i_2 \le d$.

Then, using~\eqref{eq: tau*} we write
$$
\tau_\fp\(\gamma_{i_1},\gamma_{i_2}\) p^{s + \rf{\frac{2w}{d(d+1)} }-\beta\(\gamma_{i_1},\gamma_{i_2}\)} \mid 
\tau_s = \tau_* p^{s-\beta_*} .
$$
As both $\tau$'s is prime to $p$, we see that
$$
\left\lceil \frac{2w}{d(d+1)} \right\rceil
\leq
\beta\(\gamma_{i_1},\gamma_{i_2}\)-\beta_*,
$$
which contradicts the definition of $w$ and thus concludes the proof. 
\end{proof}

\begin{rem} It is important to observe that the parameter $w$ defined by~\eqref{eq:w} in Lemma~\ref{lem:nondiv} 
does not depend on the sequence $(u_n)$ in~\eqref{eq:LRS=PowSum} with 
a proper pair $\(\vu, \vv\)$.
\end{rem}

\begin{lemma}
\label{lem:expansion} 
For any proper pair of vectors $(\vu, \vv) \in \cR_t^d \times  \cR_t^d$ for
the corresponding sequence $(u_n)$ given by~\eqref{eq:LRS=PowSum} the following holds. 
For any integers $m, n \ge 0$ and $t \ge s \ge s_*$ with 
$$
r \le p^s
$$
where
$$
r = \fl{t/s} ,
$$
 we have 
$$
 u_{n+\tau_s m} \equiv \frac{1}{r! \det A}   \sum_{j=0}^r   H_{n,j} p^{sj} m^j \pmod  {p^t}
 $$
with $H_{n,j}\in \Z$, $j =0, \ldots, r$ such that  for $j\in [0, r-d+1]$
$$
\min  \left\{\nu_p\(H_{n,j}\), \ldots,  \nu_p\(H_{n,j+d-1}\)\right\} < w+\nu_p(r!),
$$
where $w$ is defined by \eqref{eq:w}.  
\end{lemma}

\begin{proof}
Let $s\geq s_*$ defined by \eqref{eq: tau*}.

Write the binomial coefficient $\binom{m}{i}$ as a polynomial in $m$,
$$
\binom{m}{i}= \frac{c_{i,i}m^i + c_{i,i-1}m^{i-1}+\ldots +c_{i,0}}{i!} \quad \text{with } c_{i,i}, \ldots, c_{i,0}\in \Z,  
$$
and note that $c_{i,i} = 1$.
Then by~\eqref{eq:repr}, we have
\begin{align*}
 u_{n+\tau_s m} \equiv& \frac{1}{\det A}   \sum_{i=0}^r   h_{n,i} p^{si} \binom{m}{i} \\
\equiv& \frac{1}{r!\det A}   \sum_{i=0}^r  \frac{r!}{i!} h_{n,i} p^{si} \sum_{j=0}^{i}c_{i,j}m^j \\
\equiv& \frac{1}{r!\det A}   \sum_{j=0}^r   \left(\sum_{i=j}^{r}h_{n,i}\frac{r!}{i!} c_{i,j}p^{s(i-j)}\right)  p^{sj} m^j \pmod{p^t}
\end{align*}
(where the last two congruences are actually equalities).

Put
$$
H_{n,j}=\sum_{i=j}^{r}h_{n,i}\frac{r!}{i!} c_{i,j}p^{s(i-j)}.
$$

Let us fix $j\in[0,r-d+1]$ and let $k\in [j,j+d-1]$ be such that
$$
\nu_p(h_{n,k})=\min  \left\{\nu_p\(h_{n,j}\), \ldots,  \nu_p\(h_{n,j+d-1}\)\right\}.
$$
 By Lemma~\ref{lem:nondiv}, we have $\nu_p(h_{n,k})< w$. For such index, as $c_{k,k}=1$, and $r\leq p^s$,  we have
 $$
\nu_p\left(\frac{r!}{k!}c_{k,k}  \right) = \nu_p\(\frac{r!}{k!}\)< \nu_p\left(\frac{r!}{i!}c_{i,k}p^{s(i-k)}\right),
$$
for $i=k, \ldots, r$, whence 
$$
\nu_p(H_{n,k})=\nu_p( h_{n,k}) + \nu_p(r!/k!)< w+  \nu_p(r!)
$$
and the result follows.  
\end{proof}

\section{Background on exponential sums} 

\subsection{Explicit form of the Vinogradov mean value theorem}

Let $N_{k,r}(M)$ be the number of integral solutions of the system of equations
\begin{align*}
 x_1^j+\ldots+x_k^j&=y_1^j+\ldots+y_k^j, \qquad j =1, \ldots, r,\\
 1\leq x_i&,y_i\leq M,  \qquad i =1, \ldots, k.
\end{align*}

Our application of Lemma~\ref{lem:Kor} below rests on a version of the 
Vinogradov mean value theorem which  gives a precise bound on $N_{k,r}(M)$.
We use its fully explicit 
version given by Ford~\cite[Theorem~3]{Ford}, which we present 
here in the following weakened and simplified form, adjusted to our needs. 

\begin{lemma}
\label{lem:ford} There is  absolute constant $c_0>0$  such that for 
$r \ge c_0 d$ with $d \ge 2$ and $k = \fl{6 r^2 \log d}$ we have 
for every  any integer $M\ge 1$ that
$$
N_{k,r}(M)\le  r^{3r^3}M^{2k-r(r+1)/2 + \delta_{r} r^2},
$$
for some 
$$
\delta_r \le \frac{1}{1000 d}  . 
$$ 
\end{lemma} 

\begin{proof} First we make sure that $c_0$ is large enough so $r \ge 1000$ and so~\cite[Theorem~3]{Ford}
applies.  Next, examining the coefficient (depending only on $k$ and $r$ in our notation)
which appears in the bound of~\cite[Theorem~3]{Ford} we see that 
it is a product of several terms, the largest is $r^{2.055 r^3}$. Hence, for a sufficiently large $c_0$, 
we can bound this coefficient as $r^{3r^3}$.  Now it only remains to estimate $\delta_r$. 

By~\cite[Theorem~3]{Ford}, the desired result holds with 
\begin{align*}
\delta_r &\leq \frac{3}{8} \exp(1/2-2k/r^2+1.7/r) \\& \leq  \frac{3}{8} \exp(1/2-2(6 r^2 \log d-1)/r^2+1.7/r)\\
&\le   \frac{3}{8} \exp(2/3- 12 \log d)
\end{align*}
since $r \ge 1000$. It remains to notice that 
$$
  \frac{3}{8} \exp(2/3) \le 1 \mand   \exp(- 12 \log d) = d^{-12} \le \frac{1}{1000 d}
  $$
for $d \ge 2$. 
\end{proof}

We note that the recent striking advances  in the Vinogradov mean value theorem
due to 
Bourgain, Demeter and Guth~\cite{BDG} and Wooley~\cite{Wool}  are not suitable
for our purposes here as they contain implicit constants
that depend on $k$ and $n$, while in our approach   $k$ and $n$
grow together  with $M$.

\subsection{Double exponential sums with polynomials}

Our main tool to bound the exponential sum $S_{p^t}(N; \vu, \vv)$ is the following 
result of Korobov~\cite[Lemma~3]{Kor}.

\begin{lemma}\label{lem:Kor}
Assume that
$$
\left| \alpha_\ell -\frac{a_\ell}{q_\ell} \right| \le \frac{1}{q_\ell^2} \mand \gcd(a_\ell,q_\ell)=1, 
$$
for some real $\alpha_\ell$ and  integers $a_\ell,q_\ell$, $\ell=1,\ldots, r$. Then
for the sum
$$
S=\sum_{x,y=1}^M \e\(\alpha_1xy+\ldots+\alpha_n x^ry^r \)
$$
we have 
\begin{align*}
|S|^{2k^2}\leq \(64k^2\log(3Q)\)^{r/2} &
M^{4k^2-2k}N_{k,r}(M)\\
& \prod_{\ell=1}^r \min\left\{M^\ell,\sqrt{q_\ell}+\frac{M^\ell}{\sqrt{q_\ell}} \right\},
\end{align*}
where 
$$
Q=\max\{q_\ell~:~1\le \ell \le r\}.
$$
\end{lemma}

We also need the following simple result which allows us to reduce single sums 
to double sums. 

\begin{lemma}\label{lemma:sum}
 Let $f:\mathbb{R}\rightarrow\mathbb{R}$ be an arbitrary function. Then for any  integers $M,N\ge 1$ and $a\ge 0$, we have
 $$
 \left|\sum_{x=0}^{N-1}\e(f(x)) \right|\leq \frac{1}{M^2}\sum_{x=0}^{N-1}\left|\sum_{y,z=1}^{M}\e(f(x+ayz))\right|+2 a M^2.
 $$
\end{lemma}

\begin{proof} Examining the non-overlapping parts of the sums below, we see that for any positive integers $y$ and $z$
$$
 \left|\sum_{x=0}^{N-1}\e(f(x)) - \sum_{x=0}^{N-1} \e(f(x+ayz))\right|  \le 2ayz.
$$
Hence 
$$
 \left|M^2 \sum_{x=0}^{N-1}\e(f(x)) - \sum_{y,z=1}^{M} \sum_{x=0}^{N-1} \e(f(x+ayz))\right|  \le 2a\sum_{y,z=1}^{M} yz \le 2 a M^4.
$$
Changing the order of summation and using the triangle inequality, 
 the result  follows. 
\end{proof}

\subsection{Discrepancy and the Koksma-Sz\"usz inequality}
\label{sec:discrepancy} 
One of the basic tools used to study uniformity of
distribution is the celebrated
\emph{Koksma--Sz\"usz inequality\/}~\cite{Kok,Sz} 
(see also Drmota and Tichy~\cite[Theorem~1.21]{DrTi}),
which links the discrepancy of a sequence of points to certain
exponential sums.

In the notations of  Section~\ref{sec:results}, 
the Koksma--Sz\"usz inequality can be stated as follows.

\begin{lemma}\label{lem:K-S}
For any integer $V \ge 1$, we have
$$
 D_{\fu}(N) \ll \frac{1}{V}
+\frac1N\sum_{0<\|\vv\|\le V}\frac{1}{r(\vv)}
\left|S_{p^t}(N; \vu,\vv)\right|,
$$
where 
$$
\|\vv\|=  \max_{j=1, \ldots, d}|v_j|, \qquad
r(\vv) =  \prod_{j=1}^d \max\(|v_j|,1\),
$$
and the sum is taken over all vectors
$\vv= (v_1,\ldots,v_d)\in\Z^d$ with $0<\|\vv\|\le V$.
\end{lemma}

\section{Proofs of the main results} 

\subsection{Preliminary assumption}
We recall that in the case of $d=1$
results of the similar strength as of 
Theorems~\ref{thm:main} and~\ref{thm:discrepancy}  are given by Korobov~\cite{Kor}. Thus we can assume, that $d\geq 2$.  

\subsection{Proof of Theorem~\ref{thm:main}}
\label{sec:proof exp}

We can assume, that
$$
N\geq p^{8 t^{1/2}}
$$
since otherwise the result is trivial, see \eqref{eq:MS range}.

Put
$$
s = \left\lfloor \frac{\log N }{4  \log p}\right\rfloor,   \quad r= \left\lfloor \frac{t}{s}\right\rfloor \quad \text{and} \quad k=\left\lfloor 6r^2 \log d\right\rfloor.
$$
Then we have
\begin{equation}
\label{eq:assumptions}
 r<s, \qquad  \max\{w, s^*\} \leq s,\qquad p^{s}\leq N^{1/4}
\end{equation}
if $N$ is large enough, where $s^*$ and $w$ are defined by~\eqref{eq: tau*} and~\eqref{eq:w}, respectively.

First suppose, that $r\geq c_0 d$, where $c_0$ is defined in Lemma~\ref{lem:ford}. Increasing $c_0$ if necessary, we can assume
\begin{equation}
\label{eq:large r} 
r\geq 10 d.
\end{equation} 

Write 
$$
u_n=\vv A^n \vu.
$$
Then by Lemmas~\ref{lem:expansion} and \ref{lemma:sum} we have that
\begin{align*} 
|S_{p^t}(N;\vu,\vv)|&\leq \frac{1}{p^{2s}}\sum_{n=0}^{N-1}    \left| \sum_{x,y=1}^{p^s} \e\left(\frac{u_{n+\tau_s x y}}{p^t} \right)\right| +2\tau_s p^{2s}\\
&\leq \frac{1}{p^{2s}}\sum_{n=0}^{N-1}    \left| \sum_{x,y=1}^{p^s} 
\e\left(
\frac{1}{p^t \, r! \det A}   \sum_{j=0}^r   H_{n,j} p^{sj} (xy)^j 
\right)
\right| 
+2 p^{3s} . 
\end{align*}
Recalling~\eqref{eq:assumptions}, we conclude that 
\begin{equation}\label{eq:split}
|S_{p^t}(N;\vu,\vv)| \leq\frac{1}{p^{2s}}\sum_{n=0}^{N-1}    \left| \sigma_n
\right| +2N^{3/4},
\end{equation}
where
$$
\sigma_n  =\sum_{x,y=1}^{p^s}\e(f_n(x,y))
$$
and 
$$
f_n(x,y)= \frac{1}{p^t \, r! \det A}   \sum_{j=0}^r   H_{n,j} p^{sj} (xy)^j .
$$
Let us fix the index $n$ and write
$$
f_n(x,y)=\frac{a_r}{q_r}(xy)^r+\ldots + \frac{a_1}{q_1} xy+ \frac{a_0}{q_0},
$$
where the integers $a_j$ and $q_j$ are defined by
$$
\frac{H_{n,j}p^{sj}}{p^t \, r! \det A } = \frac{a_j}{q_j} \mand \gcd(a_j,q_j)=1, \qquad j=0,\ldots, r.
$$
Put
$$
\omega_j = \nu_p (H_{n,j}), \qquad j=0,\ldots, r,
$$
then
\begin{equation}\label{eq:q_j}
   p^{t-sj-\omega_j}\leq   q_j\leq r !  p^{t-sj-\omega_j}\det A, \qquad j=0,\ldots, r.
\end{equation}

By Lemma~\ref{lem:ford}, with the above choice of $k$, we have
$$
N_{k,r}(p^s)\leq r^{3r^3}p^{2sk-sr(r+1)/2 +\delta_r s r^2}
$$ 
with some   $\delta_r \leq 1/(1000 d)$.
Then by Lemma~\ref{lem:Kor},
\begin{equation}\label{eq:sigma-n}
\begin{split}
|\sigma_n|^{2k^2}\leq \(64k^2\log(3Q)\)^{r/2}  r^{3r^3} &
p^{4sk^2-sr(r+1)/2 +\delta_r s r^2 } \\
& \times   \prod_{j=1}^r \min \left\{p^{s j},\sqrt{q_j}+\frac{p^{sj}}{\sqrt{q_j}} \right\},
\end{split} \end{equation} 
where 
$$
Q=\max\{q_j:~1\le j \le r\}.
$$
By \eqref{eq:q_j}, we have $Q\leq r!  p^t\det A$ and thus
\begin{equation}\label{eq:log 3Q}
\log(3Q)\leq \log\left(3  r!  p^t  \det A\right) \leq tr \log (3 r p \det A).
 \end{equation}  
 By the choice of $r$, we have $s(r-1) < t \leq sr$. For $j$ with
$$
\frac{r+1}{2}\leq j < r
$$
we have by~\eqref{eq:q_j} that
$$
q_j \leq r!  p^{s(r-j)-\omega_j} \det A \leq r! p^{sj}  \det A \mand q_j\geq p^{s(r-j-1)-\omega_j} .
$$
Thus
$$
\frac{\sqrt{q_j}}{p^{sj}} + \frac{1}{\sqrt{q_j}}\leq \frac{1+r! \det A}{\sqrt{q_j}}\leq r^r p^{-\frac{s}{2}(r-j-1) + \omega_j/2} \det A .
$$
Put
$$
\lambda = \fl{\frac{r-1}{2d}}. 
$$
Then by Lemma~\ref{lem:expansion}, there are $j_1,\ldots, j_\lambda$ such that
\begin{equation}\label{eq:j_i 1}
\omega_{j_i}\leq w +\nu_p(r!)
\end{equation}
and 
\begin{equation}\label{eq:j_i 2}
 \frac{r+1}{2} + (i-1)d\leq j_i \leq \frac{r+1}{2} + id.
\end{equation} 
As 
$$
\nu_p(r!)\leq \left\lfloor\frac{r}{p}
  \right\rfloor + \left\lfloor\frac{r}{p^2}  \right\rfloor + \ldots\leq \frac{r}{p-1}\leq r,
$$
using~\eqref{eq:j_i 1}, we get
\begin{align*}
 \prod_{j=1}^r & \min\left\{p^{s j},\sqrt{q_j}+\frac{p^{sj}}{\sqrt{q_j}} \right\} \\
  &\qquad = p^{sr(r+1)/2  }     \prod_{j=1}^r \min\left\{1,\frac{\sqrt{q_{j}}}{p^{sj}}+\frac{1}{\sqrt{q_{j}}} \right\} \\
 &\qquad \leq p^{sr(r+1)/2  }     \prod_{(r+1)/2 \le j  \le r} \min\left\{1,\frac{\sqrt{q_{j}}}{p^{sj}}+\frac{1}{\sqrt{q_{j}}} \right\} \\
&\qquad  \leq  p^{sr(r+1)/2} \prod_{i=1}^{\lambda}\(r^r p^{-s(r-j_i-1)/2 + (w+r)/2}  \det A \)  .
\end{align*}

Next, by~\eqref{eq:j_i 2}, we have
$$
r-j_i-1 \ge  (r-3)/2 -id, \qquad i =1, \ldots, \lambda.
$$
We also notice that by~\eqref{eq:assumptions}
$$
(w+r)/2 \le  s. 
$$
Therefore,
\begin{align*}
 \prod_{i=1}^{\lambda} p^{-s(r-j_i-1)/2 + (w+r)/2} &\leq \prod_{i=1}^{\lambda} p^{-s((r-3)/2 -id )/2 + s}\\
& =
 p^{-\lambda s(r-3 - d(\lambda+1) )/4 + \lambda s }
\leq p^{-\lambda s r /4}\leq p^{- s r^2 /(16d)}, 
\end{align*}
since our assumption~\eqref{eq:large r} implies that $\lambda > r/(4d)$.

Then, substituting this bound and the bound~\eqref{eq:log 3Q} in~\eqref{eq:sigma-n}, we derive 
\begin{align*}
|\sigma_n&|^{2k^2}\\
&\leq \(2400 r^5 t (\log d)^2 \log (3rp \det A)\)^{r/2} r^{3r^3+r\lambda}(\det A)^{\lambda} 
p^{4sk^2-\frac{sr^2}{16 d} + \delta_r  s r^2  }\\
&\leq \(2400 r^5 t (\log d)^2 \log (3rp \det A)\)^{r/2} r^{3r^3+r\lambda} (\det A)^{\lambda}
p^{4sk^2-\frac{123}{2000d}s r^2} \\
& =    \(2400 r^5 (\log d)^2 \log (3rp \det A)\)^{r/2} r^{3r^3+r\lambda} (\det A)^{\lambda} t^{r/2}
p^{4sk^2-\frac{123}{2000d}s r^2}
\end{align*}
and therefore, using that 
$$
\frac{123}{2\cdot 6^2 \cdot 2000} > \frac{1}{1500}, 
$$
we conclude 
\begin{equation}\label{eq:sigma_n}
\begin{split}
 |\sigma_n| \ll 
\(\(1024 r^5 \log (3rp \det A)\)^{r/2} r^{3r^3+r\lambda} (\det A)^{\lambda}\)^{1/(2k^2)}   \quad&
\\ t^{r/(4k^2)} p^{2s - \frac{1}{1500d} \frac{s}{r^2 (\log d)^2}}.&
  \end{split} 
\end{equation} 

As  $r \le  s/t$, we have 
$$\frac{s}{r^2   }\geq\frac{s^3}{t^2   }\geq \frac{(\log N)^3}{512 t^2 (\log p)^3} 
= \frac{\log N}{512  \log p} \rho ^2.
$$
Hence 
$$
p^{s/(r\log d)^2}
\geq N^{\frac{1}{512 (\log d)^2 } \rho ^2 }.
$$ 
On the other hand, taking $c_0$ large enough so that $r \ge 35$, we see that 
$$
r \ge \frac{35}{36} (r+1)  \ge  \frac{35 t}{36 s} .
$$

Increasing $c$ if necessary, we can assume, that $90d \leq t/\log t$ as otherwise the result is trivial.

Therefore, 
\begin{align*}
\frac{r}{k^2} &  \leq \frac{1}{36 r^3 (\log d)^2}\leq \frac{s^3}{35 t^3(\log d)^2}\\
& \leq  \frac{1}{35 \cdot 4^3  \cdot (\log d)^2} \frac{(\log N)^3}{t^3 (\log p)^3} \le \frac{\log N}{1575\cdot 2^7 d\log t \log p (\log d)^2}\rho^2
\end{align*}
and thus
$$
t^{r/(4k^2)}\leq N^{  \frac{1}{1575\cdot 2^9 d (\log d)^2} \rho^2}
$$
and
$$
\( \(1024 r^5\log (3rp \det A)\)^{r/2} r^{3r^3+r\lambda} (\det A)^{\lambda}\)^{1/(2k^2)}
\ll 1.
$$

Hence, by~\eqref{eq:sigma_n} we have
$$
  |\sigma_n| \ll p^{2s} N^{-\eta_0 \rho^2 / (d  (\log d)^2)}
$$
for some absolute $\eta_0>0$ if $t$ is large enough.

Then it follows from~\eqref{eq:split},
$$
S_{p^t}(N; \vu, \vv)\ll N^{1- \eta_0 \rho^2 / (d(\log d)^2)} + N^{3/4}\leq N^{1- \eta \rho^2 / d (\log d)^2)}.
$$

If $r\leq c_0 d$, define
\begin{equation}\label{eq:N_0N}
N_0 = \left\lfloor N^{r/(4c_0d) } \right\rfloor \mand \rho_0 = \frac{\log N_0}{t \log p}> \frac{1}{8c_0 d}\rho.
\end{equation}

Then
$$
S_{p^t}(N; \vu, \vv) \leq \sum_{0\leq k < N/N_0} 
\left| S_{p^t}(N_0; A^{kN_0}\vu, \vv) \right|  + N_0.
$$
As the pair $(A^{kN_0}\vu, \vv )$ is also proper, applying the previous argument to the inner terms, we get
$$
S_{p^t}(N; \vu, \vv) \ll \frac{N}{N_0}N_0^{1-\eta \rho_0^2}  + N_0
\ll N^{1-\frac{\eta}{(8c_0d)^3} \frac{\rho^2}{d (\log d)^2}}
$$
by \eqref{eq:N_0N}. Thus, replacing $\eta$ by $\eta /(8c_0)^3$, we conclude the proof.

\subsection{Proof of Theorem~\ref{thm:discrepancy}} 
We take $V = \fl{N^{1/4}}$ and use the trivial 
observations that for any $\vv \in\Z^d$ with $0<\|\vv\|\le V$ as in Lemma~\ref{lem:K-S}
we have
$$
\gcd(v_1,\ldots,v_d, p^t) = p^\nu \le N^{1/4}
$$
with some integer $\nu$.  For such $\vv$ we have 
$$
S_{p^t}(N; \vu,\vv) = S_{p^{t-\nu}}\(N; \vu, p^{-\nu} \vv \).
$$

If $N\le \tau_{t-\nu}$, then we apply the bound of Theorem~\ref{thm:main}
directly to the sum $S_{p^{t-\nu}}\(N; \vu, p^{-\nu}\vv\)$ and note that 
$$
\frac{\log N}{(t -\nu)\log p} \ge \rho,
$$
hence the resulting bound is even getting stronger when $\nu$ increases. 

If $N> \tau_{t-\nu}$ then using periodicity, we write 
$$
 S_{p^{t-\nu}}\(N; \vu,\vv p^{-\nu}\)
 =Q S_{p^{t-\nu}}\(\tau_{t-\nu}; \vu,\vv p^{-\nu}\) + S_{p^{t-\nu}}\(R; \vu,\vv p^{-\nu}\), 
 $$
 where 
 $$
 Q = \fl{N/ \tau_{t-\nu}} \mand R = N - Q\tau_{t-\nu}.
 $$
 An apply Theorem~\ref{thm:main} again to each sum, where we also observe 
 that by Lemma~\ref{eq: tau*}
 $$
 \tau_{t-\nu} \gg p^{t-\nu} \ge p^{t} N^{-1/4} \gg N^{3/4}.
 $$
 Using Lemma~\ref{lem:K-S},  after simple calculations we obtain the desired result.

\section{Comments}  

Here we are mostly interested in short segments of the sequence~\eqref{eq:u=Au},
which is more challenging aspect of the problem and is also more important from the practical point 
of view. However longer segments of length $N$ which is comparable with $\tau_t$ are also of interest. 
It is quite plausible  that the argument of~\cite{Shp} can also be used to get the following bound for 
twisted sums
$$
\sum_{n=0}^{\tau_t-1}\e \left(\frac{1}{p^{t}} \vv A^n \vu\right) \e\(bn/\tau_t\) \ll \tau_t^{1-1/(d+1) + o(1)},
$$
which is similar to~\eqref{eq:full per exp} with some small loses in the exponent. 
In turn, together with the completing technique, see~\cite[Section~12.2]{IwKow}, this bound yields 
$$
S_{p^t}(N; \vu,\vv) \ll \tau_t^{1-1/(d+1) + o(1)}
$$
and a similar bound on the discrepancy.

We note that another natural class of moduli to consider besides powers of 
fixed primes are prime moduli $p$ for a large prime $p$. 
If the characteristic polynomial of $A$ splits completely modulo $p$, then the result of 
Bourgain~\cite[Proposition~3]{Bou1} allows to consider segments within the full period
 of length $p^\delta$, for any fixed $\delta>0$. It is very likely that coupled with ideas 
 of~\cite{Bou2} this can be reduced to segments of length $\exp\(c \log p/ \log \log p\)$ 
 for some absolute constant $c>0$. However these length remain much larger 
 in terms of the modulus and the period the admissible lengths in this work. 

\section*{Acknowledgement}
During the preparation of this  work L.~M. was partially supported by the Austrian Science Fund FWF Projects F 5506, which is part of the Special Research Program ”Quasi-Monte Carlo Methods: Theory and Applications and by NRDI (National Research Development and Innovation Office, Hungary) grant FK 142960, and
I.~S. by the Australian Research Council Grants DP170100786 and DP200100355.

\end{document}